\documentclass[11pt]{amsart}
\usepackage{amsmath, amsthm, amssymb}
\usepackage{amssymb}
\usepackage{amsfonts}
\usepackage{amscd}
\usepackage{mathrsfs}
\usepackage{latexsym}
\usepackage{amsmath}
\usepackage{txfonts}
\usepackage{tikz}
\usepackage{tikz-cd}
\usepackage{graphicx}
\usepackage{subfig}
\usepackage{xcolor}
\usepackage{pdfsync}
\usepackage[all, cmtip]{xy}

\newtheorem{theorem}{Theorem}[section]
\newtheorem{lemma}[theorem]{Lemma}
\newtheorem{corollary}[theorem]{Corollary}
\newtheorem{proposition}[theorem]{Proposition}

\theoremstyle{definition}

\newtheorem{example}[theorem]{Example}

\theoremstyle{remark}
\newtheorem{remark}[theorem]{Remark}
\newtheorem{conjecture}[theorem]{Conjecture}

\newcommand{\sshf}[1]{\mathcal{O}_{#1}}
\newcommand{\shf}[1]{\mathscr{#1}}

\newcommand{\prj}[1]{\mathbb{P}^{#1}}

\newcommand{\iso}{\simeq}
\newcommand{\ses}[3]{0\rightarrow#1\rightarrow#2\rightarrow#3\rightarrow{0}}

\newcommand{\paren}[1]{\left(#1\right)}

\numberwithin{equation}{section}

\begin{document}
\allowdisplaybreaks
\title{On the projective normality of double coverings over a rational surface}

\author{Biswajit Rajaguru}
\author{Lei Song}

\address{Department of Mathematics, University of Kansas, Lawrence, KS 66045}
\email{biswajit.rajaguru@ku.edu}
\address{Department of Mathematics, University of Kansas, Lawrence, KS 66045}
\email{lsong@ku.edu}


\dedicatory{}

\keywords{projective normality, double covering, adjoint divisor}

\begin{abstract}
We study the projective normality of a minimal surface $X$ which is a ramified double covering over a rational surface $S$ with $\dim|-K_S|\ge 1$. In particular Horikawa surfaces, the minimal surfaces of general type with $K^2_X=2p_g(X)-4$, are of this type, up to resolution of singularities. Let $\pi$ be the covering map from $X$ to $S$. We show that the $\mathbb{Z}_2$-invariant adjoint divisors $K_X+r\pi^*A$ are normally generated, where the integer $r\ge 3$ and $A$ is an ample divisor on $S$.
\end{abstract}

\maketitle

\section{Introduction}
Let $X$ be a smooth complex projective variety and $L$ be a very ample line bundle on $X$, inducing a closed embedding
\begin{equation*}
    \varphi: X\hookrightarrow \mathbb{P}(V),
\end{equation*}
where $V=H^0(X, L)$. A natural question concerning such an embedding is that for which $L$, is the natural map
\begin{equation*}
    H^0(\mathbb{P}(V), \sshf{\mathbb{P}(V)}(k))\rightarrow H^0(X, L^{\otimes k})
\end{equation*}
surjective for every positive integer $k$? Put it another way, for which $L$, can every member $D\in |L^{\otimes k}|$ be cut out from $X$ by a degree $k$ hypersurface in $\mathbb{P}(V)$?

If the answer to the above question is positive, then $X$ is embedded by the complete linear system $|L|$ as a projectively normal variety and $L$ is said to be \textit{normally generated}.

Normal generation is equivalent to so-called Property $(N_0)$. More generally, one can define Property $(N_p)$ for any integer $p\ge 0$. These properties prescribe the shape of a minimal graded free resolution of $R(L):=\oplus_{k\ge 0} H^0(X, L^{\otimes k})$ as an $S:=\text{Sym}H^0(X, L)$-module. We refer the reader to \cite{Lazarsfeld04} for an account of this subject.

A conjecture attributed to S. Mukai says, using the additive language of divisors, that for a smooth projective variety $X$ of dimension $n$, divisors of the form $K_X+(n+2+p)A+P$ shall satisfy Property $(N_p)$, where $K_X$ is the canonical divisor of $X$, $A$ is an ample divisor on $X$ and $P$ is a nef divisor. This was confirmed, in a stronger form, in the case that $A$ is very ample by \cite{EinLazarsfeld93}; and in the end of that paper, among other questions, the following was raised.

\begin{conjecture}\label{conjecture}
If $X$ is a smooth projective surface, and $A$ is an ample divisor on $X$, then $K_X+rA$ is normally generated for every integer $r\ge 4$.
\end{conjecture}

It is well known that if $r\ge 4$, $K_X+rA$ is very ample by \cite{Reider88}. Concerning normal generation, Conjecture \ref{conjecture} has been known to be true in several important cases: $K3$ surfaces by \cite{Mayer72}\cite{Saint74}, Abelian surfaces by \cite{Koizumi76}, elliptic ruled surfaces by \cite{Homma80}, \cite{Homma82}, and anticanonical rational surfaces by \cite{GalPurna01}. For minimal surfaces of general type, partial results are obtained, for instance, if $K^2_X\ge 2$ and $A-K_X$ is big and nef, then $K_X+rA$ is normally generated for $r\ge 2$ by \cite{Purna05}.

In this note, our main result is

\begin{theorem}\label{main result}
Let $S$ be a rational surface with $\dim|-K_S|\ge 1$. Let $\pi: X\rightarrow S$ be a ramified double covering of $S$ by a minimal surface $X$ (possibly singular). Let $L$ be a divisor on $S$ with the property that $K_S+L$ is nef and $L\cdot C\ge 3$ for any curve $C$. Then $K_X+\pi^*L$ is base point free and the natural map
\begin{equation}\label{multiplicaion map 2}
    \text{Sym}^rH^0(K_X+\pi^*L)\rightarrow  H^0(r(K_X+\pi^*L))
\end{equation}
surjects for every $r\ge 1$.
\end{theorem}

Typical examples for $X$ include Horikawa surfaces (see Section 2.1) and the $K3$ surfaces obtained by taking double cover of $\prj{2}$ branched along a smooth sextic. We remark that the condition that $K_S+L$ is nef and $L\cdot C\ge 3$ for any curve $C$ amounts to that $L^2\ge 7$ and $L\cdot C\ge 3$ for any curve $C$, see Proposition \ref{base point freeness of K+L}. Thus the following is an immediate corollary of Theorem \ref{main result}, which further presents some evidence for Conjecture \ref{conjecture}.

\begin{corollary}
Let $S$ be a rational surface with $\dim|-K_S|\ge 1$. Let $\pi: X\rightarrow S$ be a ramified double covering of $S$ by a minimal smooth surface $X$. Then for every $r\ge 3$ and ample divisor $L$, $K_X+r\pi^*L$ is very ample and $K_X+r\pi^*L$ is normally generated.\qed
\end{corollary}

The hypothesis that $r\ge 3$ in the corollary is optimal. In fact, consider a $K3$ surface $X$ which admits an irreducible curve $\Gamma$ with the arithmetic genus $p_a(\Gamma)=2$. The morphism induced by $|2\Gamma|$ factors as a degree two map $\pi$ onto $\prj{2}$ and the Veronese embedding of $\prj{2}$ into $\prj{5}$. Then the ample divisor $2\Gamma=\pi^*\sshf{\prj{2}}(2)$ is not very ample, and therefore cannot be normally generated (cf.~\cite{Saint74}).

The idea of the proof of Theorem \ref{main result} is as follows. By the projection formula, the surjectivity of (\ref{multiplicaion map 2}) can be reduced to the surjectivity of two multiplication maps on $S$, among which, the difficult one is to show
 \begin{equation*}
    H^0(K_S+B+L)\otimes H^0(K_S+L)\rightarrow H^0(2K_S+B+2L)
 \end{equation*}
surjects, where $B$ is the divisor class such that the branch locus of $\pi$ is a member of $|2B|$. Though $K_S+B$ is nef by the minimality of $X$, the divisor $B$ in general is not nef. Via an appropriate commutative diagram, the surjectivity of the above map can be reduced to the surjectivity of multiplication maps over two curves on $S$. One curve is a member of the linear system $|K_S+L|$, and the other is either the fixed part of $|-K_S|$ (in case that the fixed part is nonempty), or a member of $|-K_S|$. The fixed part of $|-K_S|$ is in general non-reduced; however its special structure enables us to proceed by induction on the summation of coefficients of its components.

Conventions and notations: we work throughout over the complex numbers $\mathbb{C}$. By \textit{surface}, we mean a smooth projective surface. A $\textit{curve}$ on a surface means an effective divisor on the surface. For a coherent sheaf $\shf{F}$ on $X$, we write $h^i(X, \shf{F})$ for the dimension of the cohomology group $H^i(X, \shf{F})$. When the context is clear, we simply write $H^i(\shf{F})$ for $H^i(X, \shf{F})$ and similarly for $h^i(\shf{F})$. We say a divisor class $D$ is \textit{effective} if $h^0(\sshf{S}(D))>0$. We do not distinguish between a divisor and its associated line bundle; this should not cause confusion.

The paper is organized as follows. In Section 2, we review Horikawa surface and two useful lemmas for the projective normality of regular surfaces. In Section 3, we study fixed curves and the positivity of adjoint divisors on an anticanonical rational surface. In Section 4, in the situation that a double covering over an anticanonical surface is minimal, we study the positivity of adjoint divisors involving the branch divisor. The proof of Theorem \ref{main result} is given in Section 5.

{\it Acknowledgments}: We are very grateful to B.~Purnaprajna for introducing us to the subject, suggesting the problem, patiently explaining his work \cite{GalPurna01}, \cite{Purna05}, and the thorough guidance throughout the project. We are also grateful to the referee for suggestions and corrections, which greatly improve the paper. L.S. would like to thank L. Ein for helping us formulate Lemma \ref{surjectiveness of multiplication map}, B.~Harbourne for patiently answering his questions and providing enlightening examples, and S-Y. Jow for valuable discussions on an early draft of the paper.

\section{preliminary}
First we briefly review Horikawa surfaces, which serve our main examples in this paper. Let $X$ be a smooth minimal surface of general type. It is well known that $K^2_X\ge 2p_g(X)-4$. In \cite{Horikawa76}, Horikawa studied the boundary case that $K^2_X=2p_g-4$. In this case $|K_X|$ is base point free, and the canonical map factors through a surjective degree two map onto a minimal surface $W$ (not necessarily smooth) of degree $p_g(X)-2$.

\[
\xymatrix{
 X\ar[d]_{\varphi} \ar[rd]^{\varphi_{|K_X|}} &  \\
W \ar@{^{(}->}[r] & \mathbb{P}\paren{H^0\paren{\sshf{X}(K_X)}}.}
\]

Let $\mathbb{F}_e$ denote $\mathbb{P}(\sshf{\prj{1}}\oplus \sshf{\prj{1}}(-e))$, the rational ruled surface with invariant $e\ge 0$, which has a unique negative section $\Gamma$ with $\Gamma^2=-e$. Let $f$ be the class of fibres over $\mathbb{P}^1$.

There are natural morphisms
\begin{equation}\label{natural maps}
    X\xrightarrow{\mu}X'\xrightarrow{\pi} S,
\end{equation}
where $X'$ is a ramified double covering of a rational surface $S$, $S$ is either $W$ or its blowing up along the singular point, and $\mu: X\rightarrow X'$ is the minimal resolution of singularities. Moreover, $f:=\pi\circ\mu: X\rightarrow S$ factors through $\varphi: X\rightarrow W$. Below is a concrete description of $X'$ and $S$.

\begin{theorem}[{\cite[Theorem 1.6]{Horikawa76}}]\label{classificaton of Horikawa surfaces}
Let $X$ be a minimal algebraic surface of general type with $K^2_X=2p_g-4$. Then $X$ is a minimal resolution of singularities of one of the following normal surfaces:
\begin{enumerate}
  \item a double covering of $\prj{2}$ with branch locus of degree $8$ ($p_g=3$),
  \item a double covering of $\prj{2}$ with branch locus of degree $10$ ($p_g=6$),
  \item a double covering of $\mathbb{F}_e$ whose branch locus is linearly equivalent to $6\Gamma+(p_g+2+3e)f$, where $p_g-1\ge\max\{e+3, 2e-3\}$ and $p_g-e$ is even.
  \item a double covering of $\mathbb{F}_{e-1}$ whose branch locus is linearly equivalent to $6\Gamma+4ef$, where $e=3, 4, 5$.
\end{enumerate}
Moreover in any cases, the branch locus has no infinitely near triple points.\qed
\end{theorem}

Next we present two lemmas which help to reduce projective normality of surfaces to that of curves. The first one is an observation by Gallego and Purnaprajna {\cite[p.~154]{GalPurna99}}.
\begin{lemma}\label{GP Lemma}
Let $X$ be a smooth variety with $H^1(\sshf{X})=0$. Let $E$ be a vector bundle and $L\iso\sshf{X}(C)$ be a base point free line bundle with the property that $H^1(E\otimes L^{-1})=0$. If the natural map $H^0(E|_C)\otimes H^0(L|_C)\rightarrow H^0(E\otimes L|_C)$ is surjective, then so is the natural map $H^0(E)\otimes H^0(L)\rightarrow H^0(E\otimes L)$.\qed
\end{lemma}

The second is Green's $H^0$-lemma \cite{Green84}.

\begin{lemma}\label{Green}
Let $C$ be a smooth projective curve and let $L$ and $M$ be line bundles on $C$. Assume that $W\subseteq H^0(L)$ is base point free and $h^1(M\otimes L^{-1})\le \dim W-2$. Then $W\otimes H^0(M)\rightarrow H^0(L\otimes M)$ is surjective.\qed
\end{lemma}

\begin{remark}\label{Green Remark}
The same argument in \cite{Green84} goes through as long as $C$ is a connected projective reduced curve (not necessarily irreducible) with invertible dualizing sheaf.
\end{remark}

\section{Fixed curves and adjoint divisors on anticanonical rational surfaces}
Given a surface $S$, $K_S$ denotes it canonical divisor class. Recall that, by Castelnuovo's criterion, a surface $S$ is rational if and only if $h^1(\sshf{S})=h^0(2K_S)=0$. A rational surface $S$ is called \textit{anticanonical} if $\dim|-K_S|=h^0(-K_S)-1\ge 0$. We will be mostly concerned with anticanonical rational surface with $\dim|-K_S|\ge 1$. In particular, any rational ruled surface $\mathbb{F}_e$ or its blowing up at less than 8 points is of this type. We start with a simple lemma.

\begin{lemma}\label{effectiveness of nef divisors}
Any nef divisor on an anticanonical rational surface is effective.
\end{lemma}
\begin{proof}
This follows from Riemann-Roch theorem, see \cite[Corollary II.3]{Harbourne97}.
\end{proof}

The proposition below concerns fixed curves on an anticanonical rational surface, and shall be well known to experts.

\begin{proposition}\label{H^1 of structure sheaf}
Let $S$ be a rational surface with $\dim |-K_S|\ge 1$. Suppose that $C$ is a curve on $S$ with $h^0(\sshf{S}(C))=1$. Then $h^1(C, \sshf{C})=0$.
\end{proposition}
\begin{proof}
The exact sequence
\begin{equation*}
    \ses{\sshf{S}(-C)}{\sshf{S}}{\sshf{C}}
\end{equation*}
yields the long exact sequence
\begin{equation*}
    \cdots\rightarrow H^1(\sshf{S})\rightarrow H^1(\sshf{C})\rightarrow H^2(\sshf{S}(-C))\rightarrow H^2(\sshf{S})\rightarrow 0.
\end{equation*}
Since $S$ is rational, $h^1(\sshf{S})=h^2(\sshf{S})=0$. By Serre duality, $h^2(\sshf{S}(-C))=h^0(\sshf{S}(K_S+C))$. Since $C=(K_S+C)+(-K_S)$ and $h^0(\sshf{S}(-K_S))\ge 2$, we get that if $h^0(\sshf{S}(K_S+C))\ge 1$, then $h^0(\sshf{S}(C))\ge 2$, contradicting to the assumption. Therefore we deduce that $h^0(\sshf{S}(K_S+C))=0$, and it follows that $H^1(C, \sshf{C})=0$.
\end{proof}

\begin{remark}
According to \cite[Theorem 1.7]{Artin62}, any curve $C$ on a smooth surface with $h^1(\sshf{C})=0$ has the property that for each sub-curve $C'\subseteq C$, $h^1(C', \sshf{C'})=0$. This implies that $C$ is a chain of rational curves and any intersection of its components is transverse. This type of curves appear in many contexts, for instance, if $\pi: \tilde{V}\rightarrow V$ is a resolution of an isolated rational singularities of a singular surface $V$ with exceptional integral curves $C_i$, then any $C=\sum a_iC_i$ with $a_i\ge 0$ but not all zero satisfies the property that $h^1(C, \sshf{C})=0$, cf.~\cite{Artin66}.
\end{remark}

If the anticanonical linear system $|-K_S|$ has the fixed part $F$ ($F\neq \emptyset$), we consider its decomposition
\begin{equation*}
    |-K_S|=F+|M|,
\end{equation*}
where $|M|$ is the moving part. Note that $M$ is nef, and hence effective by Lemma \ref{effectiveness of nef divisors}.

The example below is communicated to us by Harbourne, indicating that in general $F$ is neither irreducible nor reduced.
\begin{example}
Fix a line $L$ on $\prj{2}$ and four distinct points $p_1, p_2, p_3, p_4$ on $L$. Blowing up $\prj{2}$ along all $p_i$, we get $\mu: X\rightarrow \prj{2}$ with the exceptional divisors $E_{p_i}$. Then for each $i$, take a point $q_i$ on $E_{p_i}$ avoiding the strict transform $\widetilde{L}$ of $L$. Blowing up along all $q_i$, we get $\varphi: S\rightarrow X$ with the exceptional divisors $E_{q_i}$. Put $\pi: =\mu\circ\varphi$ and $\overline{L}$ for the strict transform of $\widetilde{L}$ and $N_i$ for the strict transforms of $E_{p_i}$. Then $K^2_S=1$ and the fixed part $F$ of $|-K_S|$ is given by
\begin{equation*}
    F=2\overline{L}+N_1+N_2+N_3+N_4.
\end{equation*}
\end{example}
In Section 5, we need to deal with multiplication maps on $F$. The non-reducedness of $F$ causes a technical difficulty. Thanks to Proposition \ref{H^1 of structure sheaf} and its following remark, we are able to get around it by proceeding by an induction, see Lemma \ref{surjectiveness of multiplication map}.

\begin{lemma}\label{lower bound of M.F}
Let $S$ be an anticanonical rational surface with $\dim|-K_S|\ge 1$. Suppose that $|-K_S|$ has the fixed part $F$ and write $|-K_S|=F+|M|$, where $|M|$ is the moving part. Then $M\cdot F\ge 2$.
\end{lemma}

\begin{proof}
Since $\dim|-K_S|\ge 1$, we have that $M\neq 0$. From the exact sequence $\ses{\sshf{S}(-F)}{\sshf{S}}{\sshf{F}}$, we see that
\begin{eqnarray*}
   h^0(\sshf{F})&=& 1+ h^1(\sshf{S}(-F))\\
   &=& 1+ h^1(\sshf{S}(-M))\quad\quad\text{by Serre duality}\\
   &=& 1-\chi(\sshf{S}(-M))\quad\quad \text{ note } h^2(\sshf{S}(-M))=h^0(\sshf{S}(-F))=0\\
   &=& -\frac{(-M)\cdot(-M-K_S)}{2}\\
   &=&\frac{M\cdot F}{2}.
\end{eqnarray*}
This implies that $M\cdot F\ge 2$.
\end{proof}

\begin{corollary}\label{vanishing of H^1}
Let $S$ be an anticanonical rational surface with non trivial fixed part of $|-K_S|$. With notations as above. Suppose $B$ is a divisor on $S$ such that $K_S+B$ is nef. Then $H^1(B-F)=0$.
\end{corollary}
\begin{proof}
We have $B-F=(K_S+B)+M$ is nef, and by Lemma \ref{lower bound of M.F}, $(B-F)\cdot(-K_S)\ge M\cdot (-K_S)\ge 2$. Then it follows that $H^1(B-F)=0$ from \cite[Theorem III.1]{Harbourne97}.
\end{proof}

We conclude the section with two propositions concerning the positivity of adjoint divisors.

\begin{proposition}\label{inequality1}
Let $S$ be an anticanonical rational surface and $L$ be a divisor on $S$ with the property that $L\cdot C\ge 3$ for any curve $C$ on $S$. Then $(K_S+L)\cdot(-K_S)\ge 3$ unless
\begin{enumerate}
  \item $S=\prj{2}$, $L=\sshf{\prj{2}}(3)$,
  \item $K^2_S=1, L=-3K_S$.
\end{enumerate}
In case (2), $-K_S$ is ample and has a unique base point.
\end{proposition}

\begin{proof}
Note to begin with that $L$ is effective because of nefness of $L$ and $S$ being rational. Thus $L^2\ge 3$; by Nakai-Moishezon criterion, $L$ is actually ample.
As $L\cdot (-K_S)\ge 3$, we can assume that $K^2_S>0$. Next we discuss by cases.

Case (1): $K^2_S=9$. Then $S=\mathbb{P}^2$, and $L=\sshf{\mathbb{P}^2}(m)$ for some $m\ge 3$. It is evident that $(K_S+L)\cdot(-K_S)\ge 3$ except for $L=-K_S=\sshf{\mathbb{P}^2}(3)$.

Case (2): $K^2_S=8$. Then $S=\mathbb{F}_e$ for some $e\ge 0$, and $L=a\Gamma+bf$ with $a\ge 3, b\ge ae+3$. It follows that
  \begin{eqnarray*}
    (K_S+L)\cdot(-K_S)&=&(a\Gamma+bf)\cdot(2\Gamma+(e+2)f)-8\\
                      &=&2(b-ae)+a(e+2)-8\\
                      &\ge& 4.
  \end{eqnarray*}

Case (3): $0< K^2_S\le 7$. According to the analysis in \cite[Proposition 1.10]{GalPurna01}, for ample divisor $L$, $(K_S+L)\cdot(-K_S)\ge 3$ unless one of the following exceptions occurs
  \begin{enumerate}
    \item[(a)] $L=-K_S$,
    \item[(b)] $K^2_S=1$ and $L=-2K_S$,
    \item[(c)] $K^2_S=1$ and $L=-3K_S$,
    \item[(d)] $K^2_S=2$ and $L=-2K_S$,
    \item[(e)] $K_S+L$ is base point free, $L$ is very ample and $(S, L)$ is a conic fibration over $\mathbb{P}^1$ under $|K_S+L|$.
  \end{enumerate}

Note that (b) is impossible, because $L\cdot (-K_S)\ge 3$. For (a) and (d), note that $S$ is a del Pezzo surface of degree $\le 7$, so there exists a (-1)-rational curve $C\subset S$. It follows that $L\cdot C\le 2(-K_S)\cdot C=2$. A contradiction to the assumption. For (e), we have $(K_S+L)^2=0$ and $(K_S+L)\cdot(-K_S)=2$, which imply that $L\cdot(K_S+L)=2$. Therefore $K_S+L$ cannot be effective. Using Riemann-Roch theorem and Kodaira vanishing, we obtain
\begin{equation*}
    0=\chi(K_S+L)=\frac{(K_S+L)\cdot L}{2}+1=2,
\end{equation*}
which is absurd. This shows that (e) does not occur either.

In the exceptional case (c), result on base points of $|-K_S|$ follows from \cite[Theorem III.1 (b)]{Harbourne97}.
\end{proof}

\begin{proposition}\label{base point freeness of K+L}
Let $S$ be an anticanonical rational surface and $L$ be a divisor on $S$ with the property that $L\cdot C\ge 3$ for any curve $C$ on $S$. Then the following are equivalent:

\begin{enumerate}
  \item $K_S+L$ is nef;
  \item $K_S+L$ is base point free;
  \item $L^2\ge 7$.
\end{enumerate}
 Moreover if any of the equivalent conditions holds, then $K_S+L$ is ample unless $K_S+L=0$.
\end{proposition}
\begin{proof}
(2)$\Rightarrow$(1) is obvious. (3)$\Rightarrow$(2) by \cite[Thereom 1]{Reider88}. (1)$\Rightarrow$(3):  Since $K_S+L$ is nef, hence effective, so $(K_S+L)\cdot L\ge 4$, as $(K_S+L)\cdot L$ is an even number. Then $L^2-3\ge (K_S+L)\cdot L\ge 4$, so $L^2\ge 7$.

Suppose from now on that $K_S+L\neq 0$. We first show that $(K_S+L)^2>0$. By Bertini's theorem, a general element $C\in |K_S+L|$ is smooth. Write $C=\sum^k_{i=1} C_i$, where $C_i$ are nonsingular components. Suppose to the contrary that $(K_S+L)^2=0$. Then $C^2_i=(K_S+L)\cdot C_i=0$ for each $i$. Therefore $K_S\cdot C_i=-L\cdot C_i\le -3$. By adjunction, $2g(C_i)-2=C^2_i+K_S\cdot C_i\le -3$, which implies that $g(C_i)<0$, a contradiction. Thus $(K_S+L)^2>0$ provided that $K_S+L\neq 0$.

If there exists an integral curve $C$ on $S$ such that $(K_S+L)\cdot C=0$, then $K_S\cdot C=-L\cdot C\le -3$. On the other hand, by Hodge index theorem, $C^2<0$ as $(K_S+L)^2>0$. It follows that $2p_a(C)-2=C^2+K_S\cdot C\le -4$, this is impossible again. So $(K_S+L)\cdot C>0$ for any curve $C$. Then the ampleness of $K_S+L$ is a consequence of Nakai-Moishezon criterion.
\end{proof}

\section{Double coverings over anticanonical rational surfaces}

Throughout this section, we assume that $\pi: X\rightarrow S$ is a ramified double covering, where $X$ is a minimal surface (possibly singular) and $S$ is an anticanonical rational surface. Let $B$ be a divisor, $\Gamma\in|2B|$ be the branch locus and $R\subset X$ be the ramification locus. Then $\sshf{X}(R)=\pi^*\sshf{S}(B)$, the induced morphism $\pi: R\rightarrow \Gamma$ is an isomorphism, and $\pi_*\sshf{X}\iso \sshf{S}\oplus\sshf{S}(-B)$, cf.~\cite[p. 243]{Lazarsfeld04}. The following statements are obvious.

\begin{proposition}\label{positivity of B}
With assumptions and notations as above, we have
\begin{enumerate}
  \item $K_X=\pi^*K_S+R=\pi^*(K_S+B)$.
  \item $(K_S+B)\cdot C\ge 0$, for any curve $C\subset S$.
  \item $K^2_S+B\cdot K_S\le0$; when $K^2_S>0$, $K^2_S+B\cdot K_S<0$.
  \item $B$ is effective.
\end{enumerate}
\end{proposition}

\begin{proof}
(1) is clear. Since $X$ is minimal, $2(K_S+B)\cdot C=\pi^*(K_S+B)\cdot\pi^*C=K_X\cdot\pi^*C\ge 0$. This gives (2). (3) is a special case of (2), as $-K_S$ is effective. When $K^2_S>0$, by Hodge index theorem, $K^2_S+B\cdot K_S<0$, for otherwise
$(K_S+B)^2<0$. For (4), since $B=(K_S+B)+(-K_S)$ and $-K_S$ is effective, it suffices to show $K_S+B$ is effective. This is a consequence of Lemma \ref{effectiveness of nef divisors}.
\end{proof}

\begin{lemma}\label{positivity of K+B+L}
Let $L$ be a divisor on $S$ such that $L\cdot C\ge 2$ for any curve $C$ on $S$. Then
\begin{enumerate}
  \item $K_S+B+L$ is ample and base point free, and $H^1(r(K_S+B+L))=0$ for any $r\ge 1$.
  \item If in addition $L\cdot (-K_S)\ge 3$, then $K_S+B+L$ is very ample and normally generated.
\end{enumerate}
\end{lemma}

\begin{proof}
(1) Since $K_S+B$ is nef and $L$ is ample, we deduce that $K_S+B+L$ is ample. Base point freeness and $H^1$ vanishing follow from \cite[Theorem III.1]{Harbourne97}, because
\begin{equation*}
   (-K_S)\cdot r(K_S+B+L)=r(-K_S)\cdot(K_S+B) +r(-K_S)\cdot L\ge 2.
\end{equation*}

(2) If $L\cdot (-K_S)\ge 3$, then $(-K_S)\cdot(K_S+B+L)\ge 3$. The statement follows from the criterion for $N_p$ property on rational surfaces in \cite[Theorem 1.3]{GalPurna01}.
\end{proof}

\section{Proof of Main theorem}
We begin with a few lemmas, which will be needed in the proof of our main result.
\begin{lemma}\label{H^1 vanishing}
Let $M$ be an ample and base point free divisor on a regular surface $S$ and $C$ a curve on $S$. Assume that $h^0(\sshf{C})=1$ and $-K_S-C$ is effective.
Then
\begin{equation*}
    H^i(M-C)=0 \quad\quad \text{ for\: }i>0.
\end{equation*}
\end{lemma}

\begin{proof}
By Serre duality, $h^2(M-C)=h^0(K_S+C-M)=0$, as both $-(K_S+C)$ and $M$ are effective. The short exact sequence
\begin{equation*}
    \ses{\sshf{S}(-C)}{\sshf{S}}{\sshf{C}}
\end{equation*}
yields the long exact sequence
\begin{equation*}
  0\rightarrow H^0(\sshf{S})\rightarrow H^0(\sshf{C})\rightarrow H^1(\sshf{S}(-C))\rightarrow H^1(\sshf{S})\rightarrow\cdots,
\end{equation*}
which implies that $H^1(\sshf{S}(-C))=0$.

Now take a smooth irreducible member $\Delta\in |M|$. Tensoring the short exact sequence
\begin{equation*}
    \ses{\sshf{S}}{\sshf{S}(\Delta)}{\sshf{\Delta}(\Delta)}
\end{equation*}
with $\sshf{S}(-C)$ yields the long exact sequence
\begin{equation*}
    \cdots\rightarrow H^1(\sshf{S}(-C))\rightarrow H^1(\sshf{S}(M-C))\rightarrow H^1(\sshf{\Delta}(M-C))\rightarrow\cdots
\end{equation*}
If $\sshf{S}(-K_S-C)$ is not trivial, then
\begin{equation*}
    \deg{\sshf{\Delta}(M-C)}=(M-C)\cdot\Delta>(\Delta+K_S)\cdot\Delta=\deg{K_{\Delta}},
\end{equation*}
where the inequality comes from the ampleness of $\Delta$. Therefore $H^1(\sshf{\Delta}(M-C))=0$, and hence $H^1(\sshf{S}(M-C))=0$. Otherwise if $\sshf{S}(-K_S-C)\iso \sshf{S}$, $(M-C)|_{\Delta}=K_{\Delta}$ by adjunction, which implies that $h^1((M-C)|_{\Delta})=1$. But since $h^2(-C)=h^2(K_S)=1$ and $h^2(M-C)=0$ as proved above, it follows that $h^1(M-C)=0$.
\end{proof}

There have been several authors who studied multiplications maps of two different line bundles over reduced curves. The following lemma deals with multiplication maps over a possibly non-reduced curve $C$ with the property that $h^1(\sshf{C})=0$. We hope that it could find applications in some other problems.

\begin{lemma}\label{surjectiveness of multiplication map}
Let $S$ be a surface and $C$ a curve with $h^1(\sshf{C})=0$. Put $C=\sum a_i\Gamma_i$, where $\Gamma_i$ are irreducible components of $C$. Let $L_1, L_2$ be two divisors on $S$ such that $L_1\cdot\Gamma_i>0$ for every $i$ and $L_2$ is ample and base point free. Suppose that $-K_S-\Gamma_i$ is effective for every $i$. Then the natural map
\begin{equation}\label{multiplication map}
    H^0({L_1}|_C)\otimes H^0(L_2)\rightarrow H^0((L_1+L_2)|_C)
\end{equation}
is surjective.
\end{lemma}

\begin{proof}
We do induction on $\sum a_i$. When $\sum a_i=1$, $C$ is a smooth rational curve. Since the natural map is a composition of the two maps
\begin{equation*}
    H^0(L_1|_C)\otimes H^0(L_2)\rightarrow H^0(L_1|_C)\otimes H^0(L_2|_C)\rightarrow H^0((L_1+L_2)|_C),
\end{equation*}
and $H^1(L_2-C)=0$ by Lemma \ref{H^1 vanishing}, we deduce that the map (\ref{multiplication map}) is surjective.

For general case, since $h^1(\sshf{C})=0$, there exists a component, say $\Gamma_0$, such that $\Gamma_0\cdot C\le 1+\Gamma^2_0$, cf.~\cite[Proof of Lemma 7]{Harbourne96}. Set $C'=C-\Gamma_0$, we have $\Gamma_0\cdot C'\le 1$. The short exact sequence
\begin{equation*}
    \ses{\sshf{\Gamma_0}(-C')}{\sshf{C}}{\sshf{C'}}
\end{equation*}
yields the following commutative diagram
\begin{equation*}
\includegraphics[trim=53mm 213mm 53mm 43mm, clip]{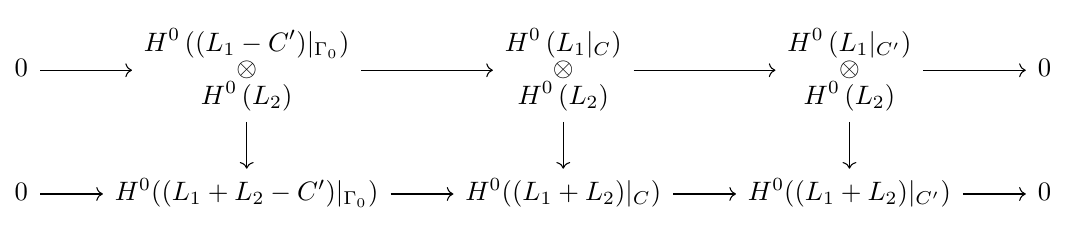}
\end{equation*}
Since $\deg(L_1-C')|_{\Gamma_0}\ge -1$, we have $H^1((L_1-C')|_{\Gamma_0})=0$; it follows that the first row is exact. Similarly so is the second one.

The surjectivity of the right column map is by induction hypothesis. The left column map factors as
\begin{eqnarray*}
    H^0\paren{(L_1-C')|_{\Gamma_0}}\otimes H^0\paren{L_2}&\xrightarrow{m_1}& H^0\paren{(L_1-C')|_{\Gamma_0}}\otimes H^0\paren{{L_2}|_{\Gamma_0}}\\
    &\xrightarrow{m_2}& H^0((L_1+L_2-C')|_{\Gamma_0}).
\end{eqnarray*}
By Lemma \ref{H^1 vanishing} again, $H^1(L_2-\Gamma_0)=0$, which implies that $m_1$ is surjective. It is evident that $m_2$ is surjective too, therefore the left column map is surjective. The statement follows from the snake lemma.
\end{proof}

\begin{remark}
The lemma implies that $H^0({L_1}|_C)\otimes H^0(L_2|_C)\rightarrow H^0((L_1+L_2)|_C)$ is surjective, but not vice versa.
\end{remark}

We record a variant of Lemma \ref{surjectiveness of multiplication map} below. Its proof is similar to that given above.
\begin{proposition}\label{surjectiveness of multiplication map'}
Let $S$ be a surface and $C=m\Gamma$, where $\Gamma$ is a smooth irreducible curve with $\Gamma^2=0$ and $m\ge 1$. Let $L_1, L_2$ be two divisors on $S$ such that $L_1\cdot\Gamma>\max\{ 2g(\Gamma), 4g(\Gamma)-L_2\cdot\Gamma-2h^1(L_2|_{\Gamma})\}$. Then the natural map
\begin{equation*}
    H^0({L_1}|_{C})\otimes H^0(L_2|_C)\rightarrow H^0((L_1+L_2)|_C)
\end{equation*}
is surjective.
\end{proposition}
\begin{proof}
Note that under our assumptions, for each $m\ge 1$, $H^1\paren{(L_1-(m-1)\Gamma)|_{\Gamma}}=0$ and the map
\begin{equation*}
    H^0\paren{(L_1-(m-1)\Gamma)|_{\Gamma}}\otimes H^0(L_2|_{\Gamma})\rightarrow H^0\paren{(L_1+L_2-(m-1)\Gamma)|_{\Gamma}}
\end{equation*}
is surjective by \cite[Proposition 2.2]{Butler94}. So the statement follows by induction on $m$.
\end{proof}

The following is a technical lemma.

\begin{lemma}\label{surjectiveness of multiplication map2}
Let $S$ be an anticanonical rational surface and $L$ be a divisor on $S$ such that $K_S+L$ big and base point free. Let $B'$ be an effective divisor on $S$ with the property that $h^1(B')=0$. Suppose that $(K_S+L)\cdot (B'-2K_S)\ge 5$. Then the natural map
 \begin{equation*}
    H^0(K_S+L+B')\otimes H^0(K_S+L)\rightarrow H^0(2K_S+2L+B')
 \end{equation*}
 is surjective.
\end{lemma}
\begin{proof}
In view of Lemma \ref{GP Lemma} and the assumption that $h^1(B')=0$, it suffices to prove the natural map
\begin{equation}\label{restricted map}
    H^0((K_S+L+B')|_{\Delta})\otimes  H^0((K_S+L)|_{\Delta})\rightarrow  H^0((2K_S+2L+B')|_{\Delta})
\end{equation}
is surjective, where $\Delta\in |K_S+L|$. Moreover $\Delta$ can be taken to be irreducible and smooth, because $K_S+L$ is big and base point free.

By Lemma \ref{Green}, the desired surjectivity of (\ref{restricted map}) will follow once the inequality
\begin{equation}\label{inequality2}
    h^1(B'|_{\Delta})\le h^0((K_S+L)|_{\Delta})-2
\end{equation}
is established.

We now analyze the involved quantities above. By adjunction,
\begin{equation*}
  \text{deg}K_{\Delta} = (\Delta+K_S)\cdot\Delta =(2K_S+L)\cdot(K_S+L).
\end{equation*}
From Riemann-Roch theorem, we have
\begin{eqnarray*}
    h^0((K_S+L)|_{\Delta})  &\ge &(K_S+L)^2+1-g(\Delta)\\
                            &=&(K_S+L)^2+1-\frac{(2K_S+L)\cdot(K_S+L)+2}{2}\\
                            &=& \frac{(K_S+L)\cdot L}{2}.
\end{eqnarray*}
On the other hand, by Serre duality, $h^1(B'|_{\Delta})=h^0((\Delta+K_S-B')|_{\Delta})$. Because $B'$ is effective, we have $B'\cdot(K_S+L)\ge 0$, and therefore
\begin{equation*}
   \deg{(\Delta+K_S-B')|_{\Delta}}=(2K_S+L-B')\cdot(K_S+L)\le\text{deg}{K_{\Delta}}.
\end{equation*}
It then follows from Clifford's theorem that
\begin{equation*}
    h^0((\Delta+K_S-B')|_{\Delta}))\le \frac{(2K_S+L-B')\cdot(K_S+L)}{2}+1.
\end{equation*}
Finally, we find
\begin{eqnarray*}
&&h^0((K_S+L)|_{\Delta})-2-h^1(B'|_{\Delta})\\
   &\ge & \frac{(K_S+L)\cdot L}{2}-2-\paren{\frac{(2K_S+L-B')\cdot(K_S+L)}{2}+1}\\
   &=&\frac{(K_S+L)\cdot (B'-2K_S)-6}{2}\\
   &\ge& -\frac{1}{2}.
\end{eqnarray*}
Thus we have established (\ref{inequality2}) and the proof is complete.
\end{proof}


We are now ready to prove the main theorem. The proof proceeds by cases depending on $|-K_S|$.

\begin{proof}[Proof of Theorem \ref{main result}]
One shall show the surjectivity of the natural map
\begin{equation}\label{surjectivity}
    H^0(r(K_X+\pi^*L))\otimes H^0(K_X+\pi^*L)\rightarrow  H^0((r+1)(K_X+\pi^*L))
\end{equation}
for all $r\ge 1$.

Note to begin with that $K_X+\pi^*L=\pi^*(K_S+B+L)$ is ample and base point free by Lemma \ref{positivity of K+B+L} (1) and the finiteness of $\pi$. On the other hand, for $r\ge 2$ and $i=1, 2$, we have
\begin{equation*}
   H^i((r-i)(K_X+\pi^*L))\iso H^i((r-i)(K_S+B+L))\oplus H^i((r-i)(K_S+L))=0
\end{equation*}
by Lemma \ref{positivity of K+B+L} (1), Kodaira vanishing and the fact that $H^2(\sshf{S})=0$. Therefore when $r\ge 2$, the surjectivity of (\ref{surjectivity}) directly follows from Castelnuovo-Mumford regularity property. It remains to treat the case $r=1$.

By pushing down to $S$, it is reduced to the surjectivity of
\begin{equation*}
    H^0(K_S+B+L)\otimes  H^0(K_S+B+L)\xrightarrow{f_1}  H^0(2K_S+2B+2L),
\end{equation*}
and
\begin{equation*}
   H^0(K_S+B+L)\otimes H^0(K_S+L)\xrightarrow{f_2} H^0(2K_S+B+2L).
\end{equation*}
By Lemma \ref{positivity of K+B+L}(2), $f_1$ is surjective.

For $f_2$, we will take a curve $C_0$ on $S$ (to be specified later) making rows in the commutative diagram exact
\begin{equation}\label{commutative diagram}
\includegraphics[trim=50mm 213mm 40mm 43mm, clip]{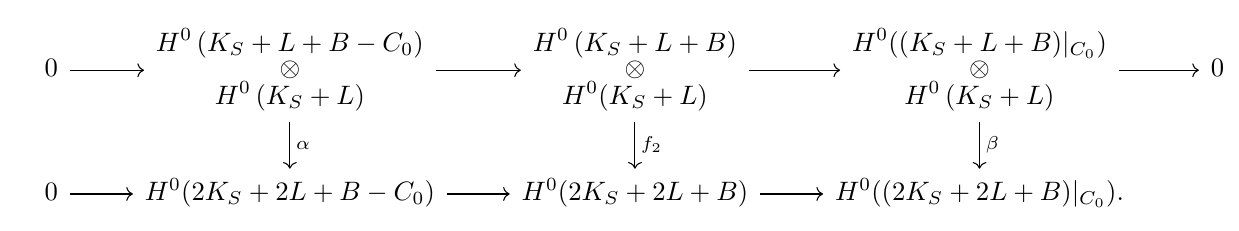}
\end{equation}
Then by the snake lemma, the surjectivity of $f_2$ will follow from the surjectivity of the column maps $\alpha$ and $\beta$.

If $K_S+L=0$, then $f_2$ is obviously surjective, so we may assume that $K_S+L\neq 0$. Thus $K_S+L$ is ample and base point free by Proposition \ref{base point freeness of K+L}. There are three cases:
\vspace{0.3cm}

Case I: $-K_S$ has the fixed part.

In this case we can take $C_0=F$, the fixed part of $|-K_S|$, and put $B'=B-C_0$. It is evident that $B'=(B+K_S)+M$ is nef. By Kodaira vanishing, $H^1(K_S+L+B')=0$, and hence the top row of commutative diagram (\ref{commutative diagram}) is exact. Because $|-K_S|$ has the fixed part, Corollary \ref{vanishing of H^1} implies that $H^1(B')=0$. Moreover by Proposition \ref{inequality1}, $(K_S+L)\cdot(-K_S)\ge 3$, so we can apply Lemma \ref{surjectiveness of multiplication map2} to deduce that $\alpha$ is surjective.

For the surjectivity of $\beta$, we will apply Lemma \ref{surjectiveness of multiplication map}. Since $C_0=F$ does not move on $S$, by Proposition \ref{H^1 of structure sheaf}, $h^1(\sshf{C_0})=0$; and it is evident that for each component $\Gamma_i$ of $C_0$, $-K_S-\Gamma_i$ is effective. Thus we can apply Lemma \ref{surjectiveness of multiplication map} by setting $L_1=K_S+L+B$ and $L_2=K_S+L$.
\vspace{0.3cm}

Case II: $-K_S$ has no fixed part and $K^2_S>0$.

In this case, we can deal with $f_2$ directly, i.e.~ take $C_0=\emptyset$. Since $-K_S$ is big and nef, $B$ is big and nef too. In particular by \cite[Theorem 8]{Harbourne96} we have that $H^1(B)=0$. By virtue of Proposition \ref{inequality1}, $(K_S+L)\cdot(-K_S)\ge 2$. Since $0\neq B$ is effective, it follows that $(K_S+L)\cdot (B-2K_S)\ge 5$, and so Lemma \ref{surjectiveness of multiplication map2} applies.
\vspace{0.3cm}

Case III: $-K_S$ has no fixed part and $K^2_S=0$.

In this case, $-K_S$ is nef, and hence base point free by \cite[Theorem III.1(c)]{Harbourne97}. Similarly $B$ is base point free. As in Case II, if $H^1(B)=0$, we are done by Lemma \ref{surjectiveness of multiplication map2}. Therefore we can assume that $H^1(B)\neq 0$, so $B\cdot (-K_S)=0$ by \cite[Theorem III.1]{Harbourne97}. But the Hodge index theorem implies that $B=m(-K_S)$ for some integer $m>0$.

Set $B_k:=k(-K_S)$. We will use commutative diagram (\ref{commutative diagram}) by taking a smooth $C_0\in |-K_S|$ and do induction on $k$ to show that for all $k\ge 0$, the map
\begin{equation*}
    \alpha_k: H^0(K_S+L+B_k)\otimes H^0(K_S+L)\rightarrow H^0(2K_S+2L+B_k)
\end{equation*}
is surjective.

When $k=0$, this is true because of \cite[Theorem 1.3]{GalPurna01} and the fact that $(K_S+L)\cdot (-K_S)\ge 3$. Suppose now that $\alpha_k$ is surjective for some $k\ge 0$. Since $H^1(K_S+L+B_k)=0$, the top row of diagram (\ref{commutative diagram}) is exact, therefore the surjectivity of $\alpha_{k+1}$ follows immediately once we have shown that
\begin{equation*}
    \beta_{k+1}: H^0((K_S+L+B_{k+1})|_{C_0})\otimes H^0(K_S+L)\rightarrow H^0((2K_S+2L+B_{k+1})|_{C_0})
\end{equation*}
is surjective.

To this end, note that
\begin{equation*}
    H^1(K_S+L-C_0)=H^1(K_S+(K_S+L))=0.
\end{equation*}
Putting $C_0=\sum_i \Gamma_i$, where $\Gamma_i$ are the smooth components, we see that for each $i$, $\Gamma^2_i=0$ and hence that $\Gamma_i\cdot (-K_S)=0$, which imply that $\Gamma_i$ is an elliptic curve. So it suffices to show that
\begin{equation*}
    H^0((K_S+L+B_{k+1})|_{\Gamma_i})\otimes H^0((K_S+L)|_{\Gamma_i})\rightarrow H^0((2K_S+2L+B_{k+1})|_{\Gamma_i})
\end{equation*}
is surjective for each $i$. This is elementary, because
\begin{equation*}
    \deg\paren{(K_S+L+B_{k+1})|_{\Gamma_i}}=\deg\paren{(K_S+L)|_{\Gamma_i}}=L\cdot\Gamma_i\ge 3=2g(\Gamma_i)+1.
\end{equation*}
This completes the proof of Case III and hence that of the theorem.
\end{proof}

We now turn to Horikawa surfaces.

\begin{corollary}
Let $X$ be a Horikawa surface and $f: X\rightarrow S$ be the canonical map (cf.~Section 2). Then for any ample divisor $A$ on $S$ and $r\ge 3$, $K_X+rf^*A$ is base point free and the image of $X$ under $|K_X+rf^*A|$ is projectively normal.
\end{corollary}
\begin{proof}
In (\ref{natural maps}), $X'$ is normal, Gorenstein and has canonical singularities, see \cite[Lemma 1.3]{Horikawa76}.
Therefore $K_X=\mu^*K_{X'}$, as $K_X$ is $\mu$-nef. It follows from the projection formula that $H^0(r(K_X+f^*L))\iso H^0(r(K_{X'}+\pi^*L))$ for any $r\ge 1$ and divisor $L$. For $r\ge 3$, since $S$ is an anticanonical rational surface, we apply Theorem \ref{main result} to $X'$, deducing the projective normality of $K_{X'}+r\pi^*A$ and hence that of $K_X+rf^*A$.
\end{proof}

\bibliography{Ontheprojectivenormalityofdoublecoveringsoverarationalsurface}{}

\begin{thebibliography}{Hom82}

\bibitem[Art62]{Artin62}
Michael Artin.
\newblock Some numerical criteria for contractability of curves on algebraic
  surfaces.
\newblock {\em American Journal of Mathematics}, pages 485--496, 1962.

\bibitem[Art66]{Artin66}
Michael Artin.
\newblock On isolated rational singularities of surfaces.
\newblock {\em American Journal of Mathematics}, 88:129--136, 1966.

\bibitem[But94]{Butler94}
David~C Butler.
\newblock Normal generation of vector bundles over a curve.
\newblock {\em Journal of Differential Geometry}, 39(1):1--34, 1994.

\bibitem[EL93]{EinLazarsfeld93}
Lawrence Ein and Robert Lazarsfeld.
\newblock Syzygies and {K}oszul cohomology of smooth projective varieties of
  arbitrary dimension.
\newblock {\em Inventiones mathematicae}, 111(1):51--67, 1993.

\bibitem[GP99]{GalPurna99}
Francisco~J. Gallego and Bangere~P. Purnaprajna.
\newblock Projective normality and syzygies of algebraic surfaces.
\newblock {\em J. Reine Angew. Math.}, 1999(506):145--180, 1999.

\bibitem[GP01]{GalPurna01}
Francisco~J. Gallego and Bangere~P. Purnaprajna.
\newblock Some results on rational surfaces and {F}ano varieties.
\newblock {\em J. Reine Angew. Math.}, 538:25--55, 2001.

\bibitem[Gre84]{Green84}
Mark Green.
\newblock Koszul cohomology and the geometry of projective varieties.
\newblock {\em Journal of Differential Geometry}, 19(1):125--171, 1984.

\bibitem[Har96]{Harbourne96}
Brian Harbourne.
\newblock Rational surfaces with \textit{K}$^2>0$.
\newblock {\em Proceedings of the American Mathematical Society},
  124(3):727--733, 1996.

\bibitem[Har97]{Harbourne97}
Brian Harbourne.
\newblock Anticanonical rational surfaces.
\newblock {\em Transactions of the American Mathematical Society},
  349(3):1191--1208, 1997.

\bibitem[Hom80]{Homma80}
Yuko Homma.
\newblock Projective normality and the defining equations of ample invertible
  sheaves on elliptic ruled surfaces with $e\ge 0$.
\newblock {\em Natural Sci. Rep. Ochanomizu Univ.}, 31(2):61--73, 1980.

\bibitem[Hom82]{Homma82}
Yuko Homma.
\newblock Projective normality and the defining equations of an elliptic ruled
  surface with negative invariants.
\newblock {\em Natural Sci. Rep. Ochanomizu Univ.}, 33(1--2):17--26, 1982.

\bibitem[Hor76]{Horikawa76}
Eiji Horikawa.
\newblock Algebraic surfaces of general type with small $c^2_1$, {I}.
\newblock {\em Annals of Mathematics}, 104(2):357--387, 1976.

\bibitem[Koi76]{Koizumi76}
Shoji Koizumi.
\newblock Theta relations and projective normality of {a}belian varieties.
\newblock {\em American Journal of Mathematics}, pages 865--889, 1976.

\bibitem[Laz04]{Lazarsfeld04}
R.~Lazarsfeld.
\newblock {\em Positivity in algebraic geometry {I}: {C}lassical setting: line
  bundles and linear series}, volume~48.
\newblock Springer, 2004.

\bibitem[May72]{Mayer72}
Alan~L. Mayer.
\newblock Families of {K}-3 surfaces.
\newblock {\em Nagoya Math. J.}, 48:1--17, 1972.

\bibitem[Pur05]{Purna05}
Bangere~P. Purnaprajna.
\newblock Some results on surfaces of general type.
\newblock {\em Canadian J. Math.}, 57:724--749, 2005.

\bibitem[Rei88]{Reider88}
Igor Reider.
\newblock Vector bundles of rank 2 and linear systems on algebraic surfaces.
\newblock {\em Annals of Mathematics}, 127(2):309--316, 1988.

\bibitem[SD74]{Saint74}
Bernard Saint-Donat.
\newblock Projective models of {K}-3 surfaces.
\newblock {\em American Journal of Mathematics}, 96(4):602--639, 1974.

\end{thebibliography}
\bibliographystyle{alpha}

\end{document}